\documentclass[10pt,reqno]{amsart}
\usepackage{latexsym}
\usepackage{graphicx}
\usepackage{fancyhdr,amssymb}

\usepackage[normalem]{ulem} 
\usepackage{array} 

\newcommand{\mycircle}[1]{\raisebox{-2.5pt}{\huge\textcircled{\raisebox{2pt}{\normalsize #1}}}} 
\newcommand{\mybigcircle}[1]{\raisebox{-2.5pt}{\Huge\textcircled{\raisebox{2pt}{\normalsize #1}}}} 

\theoremstyle{plain}
\numberwithin{equation}{section}
\newtheorem{thm}{Theorem}[section]
\newtheorem{theorem}[thm]{Theorem}

\theoremstyle{definition}

\begin{document}
\fancyhead{}
\renewcommand{\headrulewidth}{0pt}
\fancyfoot{}
\fancyfoot[LE,RO]{\medskip \thepage}
\fancyfoot[LO]{\medskip MISSOURI J.~OF MATH.~SCI., FALL 2020}
\fancyfoot[RE]{\medskip MISSOURI J.~OF MATH.~SCI., VOL.~32, NO.~2}

\setcounter{page}{1}

\title[Properties of Higher-Order Prime Number Sequences]{Properties of Higher-Order Prime Number Sequences}
\author{Michael P. May}
\address{443 W 2825 S\\
         Perry, Utah\\
         84302}
\email{mikemay@mst.edu}

\begin{abstract}
In this paper, we analyze properties of prime number sequences produced by the alternating sum of higher-order subsequences of the primes.  We also introduce a new sieve which will generate these prime number sequences via the systematic selection and elimination of prime number indexes on the real number line.
\end{abstract}

\maketitle


\section{Introduction}


We provide a brief explanation of the concept of prime numbers with prime subscripts. To this end, we refer the reader to the OEIS reference 
\cite{OEISHigherOrderPrimeNumbers} on higher-order prime numbers,
and to the paper \cite{PrimesHavingPrimeSubscripts}.
Higher-order prime numbers, also called superprime numbers or superprimes, are the primes that occupy prime-numbered positions within the sequence of all prime numbers. They are also called prime-indexed primes. The OEIS sequence A006450 \cite{OEIS_A006450} defines this sequence of prime numbers as
\begin{equation*}
a(n) = p_{p_{n}}= 3, 5, 11, 17, 31, 41, 59, 67, 83, 109, 127, 157, 179, 191,\dots \; .
\end{equation*}
That is, if $p_i$ denotes the $i$th prime number, then the numbers of sequence A006450 are of the form $p_{p_{i}}$.

Denote by \(\mathbb{P}\) the sequence of primes \(p_1, p_2, \dots\).
Suppose \(A \subseteq \mathbb{P}\), and suppose \(I\) is an index set of \(A\), that is,
for every \(n \in I\) we have \(p_n \in A\). Then we define
\(A_I = \{ p_n \mid n \in I \}\). Note that the indices are the position in the whole sequence of primes
and not in \(A\).

We define \(P^{(k)}\) recursively by 
\[
P^{(1)} = \mathbb{P} \text{ and } P^{(k)} = P_{P^{(k-1)}} \text{ for } k \ge 2.
\]
Thus \(P^{(1)}\) are the primes, \(P^{(2)}\) are the prime-indexed primes, and so on.
We use the notations \(P^{(k)}\) and \(p^{(k)}\) interchangeably.

\par
\ \\
\noindent\rule{0.84in}{0.4pt} \par
\medskip
\indent\indent {\fontsize{8pt}{9pt} \selectfont DOI: 10.35834/YYYY/VVNNPPP \par}
\indent\indent {\fontsize{8pt}{9pt} \selectfont MSC2020: 11A41, 11B99 \par}
\indent\indent {\fontsize{8pt}{9pt} \selectfont Key words and phrases: Primes, Prime number subsequences, Prime number \par} \indent\indent {\fontsize{8pt}{9pt} \selectfont indexes, Infinite sequences, Alternating sum, Higher-order primes, Sieves \par}

\thispagestyle{fancy}

\vfil\eject
\fancyhead{}
\fancyhead[CO]{\hfill HIGHER-ORDER PRIME NUMBER SUBSEQUENCES}
\fancyhead[CE]{M.~P.~MAY  \hfill}
\renewcommand{\headrulewidth}{0pt}

We will show that there is a unique partition of \(\mathbb{P}\) into an index set \(I\)
and an indexed set \(\mathbb{P}_I\). 
We denote these by \(\mathbb{P}' = I\) and \(\mathbb{P}'' = \mathbb{P}_I\).
We will see later that these two sets can be defined in terms of the sequences of higher-order primes.

\begin{theorem} \label{1.1}
There is a unique subset \(I\) of the prime numbers such that
\begin{equation*}
\mathbb{P^{}}={I}\cup\mathbb{P}_{{{I}}}
\end{equation*}
and
\begin{equation*}
{I}\cap\mathbb{P}_{{I}}=\emptyset,
\end{equation*}
\end{theorem} 

\begin{proof}
We begin by showing that ${I}$ is unique and that the first $n$ elements of ${I}$ and $\mathbb{P}_{{I}}$ determine the $n+1st$ element of ${I}$. First, we write
\begin{equation*}
q_1, \dots, q_n
\end{equation*}
to designate the smallest elements of ${I}$ in increasing order. And because the $nth$ prime function $p_n$ is a monotonic increasing function, it follows that
\begin{equation*}
p_{q_{1}},\dots,p_{q_{n}}
\end{equation*}
are the smallest elements of $\mathbb{P}_{{I}}$ in corresponding increasing order. Now let $q$ be the smallest prime number that is different from all $q_1$,\dots,$q_n$ and $p_{q_{1}}$,\dots,$p_{q_{n}}$. We know that either $q \in {I}$ or $q \in \mathbb{P}_{{I}}$, and our claim is that $q \in {I}$. But if we suppose to the contrary that $q \in \mathbb{P}_{{I}}$, then there must be a prime number $r$ such that $q=p_r$. Now we assumed that $q>p_{q_{n}}$, and it follows from the definition that $r>q_n$ and $r<p_r=q$. But this means that $r$ satisfies the same conditions as $q$: i.e., that $r$ is the smallest prime different from all $q_1$,\dots,$q_n$ and $p_{q_{1}}$,\dots,$p_{q_{n}}$. So we should have chosen $r$ instead of $q$ because $r$ is smaller than the $q$ we assumed to be an element of $\mathbb{P}_{{I}}$. Therefore, our assumption that $q \in \mathbb{P}_{{I}}$ is a contradiction, and the smallest prime $q$ different from all $q_1$,\dots,$q_n$ and $p_{q_{1}}$,\dots,$p_{q_{n}}$ must be an element of I because the elements of ${I}$, which serve as the subscripts of $\mathbb{P}_{{I}}$,  always lag behind the elements of $\mathbb{P}_{{I}}$. Furthermore, since all prime numbers $\mathbb{P}$ are consumed in the process of adding the smallest available prime number different from all $q_1$,\dots,$q_n$ and $p_{q_{1}}$,\dots,$p_{q_{n}}$ to ${I}$, then we have shown that ${I}$ and $\mathbb{P}_{I}$ are complement sets of  $\mathbb{P}$.
\end{proof}

A new sieve will be introduced shortly which iteratively applies the smallest available prime number as a subscript ${I}$ to the set of all prime numbers $\mathbb{P}$ so that eventually all primes get used up and the partition is created for $\mathbb{P}$.


\section{Alternating Sum}


We discovered that the unique subset of prime numbers $\mathbb{P{'}}$ can be generated via an alternating sum of prime number subsequences of increasing order.  Thus, one can define $\mathbb{P{'}}$ as
\begin{equation}
\mathbb{P^{'}}={\left\lbrace{(-1)^{n-1}}\left\lbrace{p^{(n)}}\right\rbrace\right\rbrace}_{n=1}^\infty \label{eq:1}
\end{equation}
where the right-hand side of Eq. \ref{eq:1} is an expression of the alternating sum
\begin{equation}
\left\lbrace{p^{(1)}}\right\rbrace - \left\lbrace{p^{(2)}}\right\rbrace + \left\lbrace{p^{(3)}}\right\rbrace - \left\lbrace{p^{(4)}}\right\rbrace + \left\lbrace{p^{(5)}}\right\rbrace -\dots \; \label{eq:2}
\end{equation}
where the plus and minus signs indicate set union and set difference rather than arithmetic
operations on the elements of the sequences.
The sets $\left\lbrace{p^{(n)}}\right\rbrace$ in \ref{eq:2} are defined as
\begin{equation}
\left\lbrace{p^{(1)}}\right\rbrace = {\left\lbrace{p_n}\right\rbrace}_{n=1}^\infty = \left\lbrace2,3,5,7,11,13,17,19,23,29,31,37,41,43,\dots\right\rbrace \label{eq:3}
\end{equation}
$=$ the set of all prime numbers, or what we will call the set of all first-order primes ${P}^{(1)}$(\cite{OEIS_A000040})
\begin{equation}
\left\lbrace{p^{(2)}}\right\rbrace = {\left\lbrace{p_{p_n}}\right\rbrace}_{n=1}^\infty = \left\lbrace3,5,11,17,31,41,59,67,83,109,127,157,\dots\right\rbrace \label{eq:4}
\end{equation}
$=$ the set of prime numbers with prime indexes, or second-order primes ${P}^{(2)}$ (\cite{OEIS_A006450})
\begin{equation}
\left\lbrace{p^{(3)}}\right\rbrace = {\left\lbrace{p_{p_{p_n}}}\right\rbrace}_{n=1}^\infty = \left\lbrace5,11,31,59,127,179,277,331,431,599,\dots\right\rbrace \label{eq:5}
\end{equation}
$=$ the set of second-order prime numbers with prime indexes, or third-order primes ${P}^{(3)}$ (\cite{OEIS_A038580})
\begin{equation}
\left\lbrace{p^{(4)}}\right\rbrace = {\left\lbrace{p_{p_{p_{p_n}}}}\right\rbrace}_{n=1}^\infty = \left\lbrace11,31,127,277,709,\dots\right\rbrace \label{eq:6}
\end{equation}
$=$ the set of third-order prime numbers with prime indexes, or fourth-order primes ${P}^{(4)}$  (\cite{OEIS_A049090})
\begin{equation}
\left\lbrace{p^{(5)}}\right\rbrace = {\left\lbrace{p_{p_{p_{p_{p_n}}}}}\right\rbrace}_{n=1}^\infty = \left\lbrace31,127,709,\dots\right\rbrace \label{eq:7}
\end{equation}
$=$ the set of fourth-order prime numbers with prime indexes, or fifth-order primes ${P}^{(5)}$ (\cite{OEIS_A049203})
and so on and so forth.

To facilitate the construction of the unique subset of prime numbers $\mathbb{P^{'}}$, one can circumvent the bulk alternating summation of the aforementioned prime number subsequences of increasing order by arranging those subsequences side-by-side and summing laterally across the rows to create the new $\mathbb{P^{'}}$ sequence term-by-term as follows:

\begin{table}[ht]
\centering
\begin{tabular}{cccccccccc}

(row) & $+p^{(1)}$ & $-p^{(2)}$ & $+p^{(3)}$ & $-p^{(4)}$ & $+p^{(5)}$ & $-p^{(6)}$ & $\dots$ & $p{'}$ \\ 
&  &  &  &  &  &  &  &  &  \\ 
(1) & 2 & $\longrightarrow$ & $\longrightarrow$ & $\longrightarrow$ & $\longrightarrow$ & $\longrightarrow$ & $\longrightarrow$ &  2 \\ 
(2) & 3 & 3 & $\longrightarrow$ & $\longrightarrow$ & $\longrightarrow$ & $\longrightarrow$ & $\longrightarrow$ &  0 \\ 
(3) & 5 & 5 & 5 & $\longrightarrow$ & $\longrightarrow$ & $\longrightarrow$ & $\longrightarrow$ &  5 \\ 
(4) & 7 & $\longrightarrow$ & $\longrightarrow$ & $\longrightarrow$ & $\longrightarrow$ & $\longrightarrow$ & $\longrightarrow$ &  7 \\ 
(5) & 11 & 11 & 11 & 11 & $\longrightarrow$ & $\longrightarrow$ & $\longrightarrow$ &  0 \\ 
(6) & 13 & $\longrightarrow$ & $\longrightarrow$ & $\longrightarrow$ & $\longrightarrow$ & $\longrightarrow$ & $\longrightarrow$ &  13 \\ 
(7) & 17 & 17 & $\longrightarrow$ & $\longrightarrow$ & $\longrightarrow$ & $\longrightarrow$ & $\longrightarrow$ &  0 \\ 
(8) & 19 & $\longrightarrow$ & $\longrightarrow$ & $\longrightarrow$ & $\longrightarrow$ & $\longrightarrow$ & $\longrightarrow$ &  19 \\ 
(9) & 23 & $\longrightarrow$ & $\longrightarrow$ & $\longrightarrow$ & $\longrightarrow$ & $\longrightarrow$ & $\longrightarrow$ &  23 \\ 
(10) & 29 & $\longrightarrow$ & $\longrightarrow$ & $\longrightarrow$ & $\longrightarrow$ & $\longrightarrow$ & $\longrightarrow$ &  29 \\ 
(11) & 31 & 31 & 31 & 31 & 31 & $\longrightarrow$ & $\longrightarrow$ &  31 \\ 
$\vdots$ & $\vdots$ & $\vdots$ & $\vdots$ & $\vdots$ & $\vdots$ & $\vdots$ & $\vdots$ & $\vdots$  
\end{tabular}
\caption{Calculation of \(\mathbb{P}'\) as an alternating sum.}
\label{table:sum}
\end{table}

The infinite prime number sequence $\mathbb{P^{'}}$ emerging in the rightmost column of Table \ref{table:sum} is
\begin{equation*}
\mathbb{P^{'}} = \left\lbrace {p{'}} \right\rbrace = \left\lbrace 2,5,7,13,19,23,29,31,37,43,47,53,59,61,71,\dots \right\rbrace.
\end{equation*}
This is the new OEIS sequence A333242 (\cite{OEIS_A333242}). When $\mathbb{P^{'}}$ is generated by the cross-summation method demonstrated in Table \ref{table:sum} above, this sequence is represented by the equation:
\begin{align*}
& \phantom{{}+{}} \{p_1^{(1)}-p_1^{(2)}+p_1^{(3)}-p_1^{(4)}+\cdots\}+\{p_2^{(1)}-p_2^{(2)}+p_2^{(3)}-p_2^{(4)}+\cdots\} \\
&+ \{p_3^{(1)}-p_3^{(2)}+p_3^{(3)}-p_3^{(4)}+\cdots\}+  \{p_4^{(1)}-p_4^{(2)}+p_4^{(3)}-p_4^{(4)}+\cdots\}
+ \cdots \\
&+\{p_\infty^{(1)}-p_\infty^{(2)}+p_\infty^{(3)}-p_\infty^{(4)}+\cdots\} = {\left\lbrace\displaystyle\sum\limits_{n=1}^\infty (-1)^{n-1}p_r^{(n)}\right\rbrace}_{r=1}^\infty
\end{align*}
where the superscript $(n)$ denotes the order of primeness as defined in Eqs. \ref{eq:3}--\ref{eq:7} and the subscript $r$ denotes the row number corresponding to the natural number in the leftmost column of Table \ref{table:sum}.  Therefore, we find that the new prime number sequence $\mathbb{P{'}}$ can either be generated via the bulk alternating summation of infinite prime number subsequences of increasing order or it can be constructed finitely term-by-term by summing laterally across the prime number subsequences of increasing order as illustrated in Table \ref{table:sum} so that we have
\begin{equation*}
\mathbb{P^{'}}={\left\lbrace{(-1)^{n-1}}\left\lbrace{p^{(n)}}\right\rbrace\right\rbrace_{n=1}^\infty} = {\left\lbrace\displaystyle\sum\limits_{n=1}^\infty (-1)^{n-1}p_r^{(n)}\right\rbrace_{r=1}^\infty}.
\end{equation*}
Recall that we previously defined $\mathbb{P^{''}}$ as the complement of $\mathbb{P^{'}}$ with respect to the set of all prime numbers $\mathbb{P^{}}$. As such, we have
\begin{equation}
\mathbb{P^{''}} = \overline{\mathbb{P^{'}}} = \mathbb{P{}} - \mathbb{P^{'}} = \left\lbrace3,11,17,41,67,83,109,127,\dots\right\rbrace. \label{eq:8}
\end{equation}
The right-hand side of Eq. \ref{eq:8} represents the OEIS sequence A262275 (\cite{OEIS_A262275}). We observe that $\mathbb{P^{''}}$ can also be defined in terms of the following alternating sum of prime number subsequences of increasing order by adjusting the starting index and reversing the sign of the terms on the right hand side of Eq. \ref{eq:1}:
\begin{equation*}
\mathbb{P^{''}}={\left\lbrace{(-1)^{n}}\left\lbrace{p^{(n)}}\right\rbrace\right\rbrace}_{n=2}^\infty \; = \left\lbrace{p^{(2)}}\right\rbrace - \left\lbrace{p^{(3)}}\right\rbrace + \left\lbrace{p^{(4)}}\right\rbrace - \cdots \; .
\end{equation*}
Similarly, we can define even higher classes of prime number sequences formed by the alternating sums of prime number subsequences of increasing order.  If we let
\begin{align*}
\mathbb{P}^{(1)} &= \mathbb{P^{'}} \\
\mathbb{P}^{(2)} &= \mathbb{P^{''}} \\
\mathbb{P}^{(3)} &= \mathbb{P^{'''}} \\
&\vdots
\end{align*}
then
\begin{align*}
\mathbb{P}^{(1)} &= \left\lbrace{(-1)^{n-1}}\left\lbrace{p^{(n)}}\right\rbrace\right\rbrace_{n=1}^\infty \\ 
\mathbb{P}^{(2)} &= \left\lbrace{(-1)^{n}}\left\lbrace{p^{(n)}}\right\rbrace\right\rbrace_{n=2}^\infty \\ 
\mathbb{P}^{(3)} &= \left\lbrace{(-1)^{n-1}}\left\lbrace{p^{(n)}}\right\rbrace\right\rbrace_{n=3}^\infty. \\
&\vdots
\end{align*}
We conclude this section with a proof of our initial claim that $\mathbb{P{'}}$ is equivalent to the alternating sum of all prime number sequences of increasing order beginning at $\mathbb{P}=P^{(1)}$.

\begin{theorem} \label{2.1}
\begin{equation*}
\mathbb{P^{'}} = \left\lbrace{p^{(1)}}\right\rbrace - \left\lbrace{p^{(2)}}\right\rbrace + \left\lbrace{p^{(3)}}\right\rbrace - \left\lbrace{p^{(4)}}\right\rbrace + \cdots \; .
\end{equation*}
\end{theorem}

\begin{proof}
Theorem \ref{1.1} states that 
\begin{equation*}
\mathbb{P{}}=\mathbb{P{'}}+\mathbb{P{''}}.
\end{equation*}
Since $\mathbb{P{'}}$ has been defined as the alternating sum beginning at $P^{(1)}$, and $\mathbb{P{''}}$ has been defined as the alternating sum beginning at $P^{(2)}$, then we have
\begin{equation*}
\mathbb{P^{'}} = \left\lbrace{p^{(1)}}\right\rbrace - \left\lbrace{p^{(2)}}\right\rbrace + \left\lbrace{p^{(3)}}\right\rbrace - \left\lbrace{p^{(4)}}\right\rbrace + \cdots
\end{equation*}
and
\begin{equation*}
\mathbb{P^{''}} = \left\lbrace{p^{(2)}}\right\rbrace - \left\lbrace{p^{(3)}}\right\rbrace + \left\lbrace{p^{(4)}}\right\rbrace - \left\lbrace{p^{(5)}}\right\rbrace + \cdots
\end{equation*}
so that
\begin{align*}
\mathbb{P{}}&=\mathbb{P{'}}+\mathbb{P{''}} \\
&= \left\lbrace{P^{(1)}} - {P^{(2)}} + {P^{(3)}} - {P^{(4)}} + {P^{(5)}} -\cdots \right\rbrace \\
&+ \left\lbrace{P^{(2)}} - {P^{(3)}} + {P^{(4)}} - {P^{(5)}} + {P^{(6)}} -\cdots \right\rbrace \\
&= \mathbb{P{}}. \qedhere
\end{align*} 
\end{proof}


\section{N-Sieve}


We discovered that the $\mathbb{P^{'}}$ sequence can also be constructed by a new sieving operation performed on the domain of natural numbers $\mathbb{N}$ via the sequential selection and elimination of prime number indexes on the natural number line.  Hence, we introduce a new algorithm which we will call the ``N-sieve'' for performing such an operation.  The sieve is accomplished as follows:  Starting with $n=1$, choose the prime number with subscript 1 (i.e., $p_1=2$) as the first term of the $\mathbb{P^{'}}$ sequence and eliminate that prime number from the natural number line.  
We mark eliminated items by circling them.
Then, proceed forward on $\mathbb{N}$ from $1$ to the next available natural number.  Since $2$ was eliminated from the natural number line in the previous step, one moves forward to the next available natural number that has not been eliminated, which is $3$.  $3$ then becomes the subscript for the next $\mathbb{P^{'}}$ term which is $p_3=5$, and $5$ is then eliminated from the natural number line, and so on and so forth.  Such a sieving operation has been carried out in Table \ref{table:N100} for the natural numbers $1$ to $100$ and should be easy to follow.

\begin{table}[h]
\begin{center}
\[
\renewcommand{\arraystretch}{1.8}
\begin{array}{w{c}{6pt}w{c}{6pt}w{c}{6pt}w{c}{6pt}w{c}{6pt}w{c}{6pt}w{c}{6pt}w{c}{6pt}w{c}{6pt}w{c}{6pt}w{c}{6pt}w{c}{6pt}w{c}{6pt}w{c}{6pt}w{c}{6pt}w{c}{6pt}w{c}{6pt}w{c}{6pt}w{c}{6pt}w{c}{6pt}}
 1 &  \mycircle{2} &  3 &  4 &  \mycircle{5} &  6 &  \mycircle{7} &  8 &  9 & 10 &
11 & 12 & \mycircle{13} & 14 & 15 & 16 & 17 & 18 & \mycircle{19} & 20 \\
21 & 22 & \mycircle{23} & 24 & 25 & 26 & 27 & 28 & \mycircle{29} & 30 &
\mycircle{31} & 32 & 33 & 34 & 35 & 36 & \mycircle{37} & 38 & 39 & 40 \\
41 & 42 & \mycircle{43} & 44 & 45 & 46 & \mycircle{47} & 48 & 49 & 50 &
51 & 52 & \mycircle{53} & 54 & 55 & 56 & 57 & 58 & \mycircle{59} & 60 \\
\mycircle{61} & 62 & 63 & 64 & 65 & 66 & 67 & 68 & 69 & 70 &
\mycircle{71} & 72 & \mycircle{73} & 74 & 75 & 76 & 77 & 78 & \mycircle{79} & 80 \\
81 & 82 & 83 & 84 & 85 & 86 & 87 & 88 & \mycircle{89} & 90 &
91 & 92 & 93 & 94 & 95 & 96 & \mycircle{97} & 98 & 99 & 100
\end{array}
\]
\caption{N-sieve applied to first 100 integers}\label{table:N100}
\end{center}
\end{table}

We will now designate $\mathbb{P{'}}$, which has been alternately created via the sieving operation above, by the new notation $\left\lfloor{\raisebox{-.3pt}{\!{\dashuline{\begin{math}\,\mathbb{N}\end{math}}}}}\right\rfloor$ to indicate that the natural numbers have been sieved to produce this prime number sequence. In other words,
\begin{equation*}
\left\lfloor{\raisebox{-.3pt}{\!{\dashuline{\begin{math}\,\mathbb{N}\end{math}}}}}\right\rfloor = \mathbb{P{'}} = \left\lbrace{2,5,7,13,19,23,29,31,37,43,47,53,59,61,71,\dots}\right\rbrace.
\end{equation*}
Thus, we can define the N-sieve of the natural numbers in the terms that were previously-derived for $\mathbb{P{'}}$:
\begin{equation*}
\left\lfloor{\raisebox{-.3pt}{\!{\dashuline{\begin{math}\,\mathbb{N}\end{math}}}}}\right\rfloor = \mathbb{P{'}} = {\left\lbrace{(-1)^{n-1}}\left\lbrace{p^{(n)}}\right\rbrace\right\rbrace}_{n=1}^\infty = {{\left\lbrace\displaystyle\sum\limits_{n=1}^\infty (-1)^{n-1}p_r^{(n)}\right\rbrace}_{r=1}^\infty}.
\end{equation*}
At this point, we conclude that the prime number sequence $\mathbb{P{'}}$ can be constructed in either of three different ways:
\begin{enumerate}
\item
Via the bulk alternating summation of the infinite prime number subsequences of increasing order;
\item
By arranging the prime number subsequences of increasing order side-by-side and taking sums laterally across the subsequences row-by-row to create the $\mathbb{P{'}}$ sequence iteratively term-by-term as illustrated in Table \ref{table:sum};
\item
By the new N-sieve introduced above.
\end{enumerate}


\section{Extension of the Application}


An interesting property that we observed earlier in the relationship between the set of all prime numbers $\mathbb{P{}}$ and the complement prime number sets is:
\begin{equation*}
\mathbb{P{}}-{I}=\mathbb{P}_{I},
\end{equation*}
or more clearly,
\begin{equation*}
\mathbb{P{}} - \left\lbrace2,5,7,13,19,23,29,31,\dots\right\rbrace = \left\lbrace {p_{{}_{{}_{{2}}}},p_{{}_{{}_{{5}}}},p_{{}_{{}_{{7}}}},p_{{}_{{}_{{13}}}},p_{{}_{{}_{{19}}}},p_{{}_{{}_{{23}}}},
p_{{}_{{}_{{29}}}},p_{{}_{{}_{{31}}}},\dots}\right\rbrace
\end{equation*}
where the prime numbers of the sequence ${I}$ form the indexes for the complement set of primes $\mathbb{P}_{{I}}$ such that 
\begin{equation*}
\mathbb{P}_{{{I}}} = \lbrace p_k \mid k \in {I} \rbrace.
\end{equation*}
When we take the new sieving operation to the next level by sifting the set of all prime numbers $\mathbb{P{}}=P^{(1)}$, we produce the OEIS sequence A262275 (\cite{OEIS_A262275}):

\begin{center}
\noindent 2 \mycircle{3} 5 7 \mycircle{11} 13 \mycircle{17} 19 23 29 31 37 \mycircle{41} 43 47 53 59 61 \mycircle{67} 73 \dots .
\end{center}

We will now designate $\mathbb{P{''}}$, which has been alternately created via the sieving operation above, by the notation $\left\lfloor{\raisebox{-.3pt}{\!{\dashuline{\begin{math}\,\mathbb{P}\end{math}}}}}\right\rfloor$ to indicate that the prime numbers have been N-sieved to produce the prime number sequence
\begin{equation*}
\left\lfloor{\raisebox{-.3pt}{\!{\dashuline{\begin{math}\,\mathbb{P}\end{math}}}}}\right\rfloor = \mathbb{P{''}} = \mathbb{P}_{{I}} = \left\lbrace{3,11,17,41,67,\dots}\right\rbrace.
\end{equation*}
Thus, we can now define the prime number sequence obtained by the N-sieving of the prime numbers $\mathbb{P}$ as:
\begin{equation*}
\left\lfloor{\raisebox{-.3pt}{\!{\dashuline{\begin{math}\,\mathbb{P}\end{math}}}}}\right\rfloor = \mathbb{P{''}} = {\left\lbrace{(-1)^{n}}\left\lbrace{p^{(n)}}\right\rbrace\right\rbrace}_{n=2}^\infty = {{\left\lbrace\displaystyle\sum\limits_{n=2}^\infty (-1)^{n}p_r^{(n)}\right\rbrace}_{r=1}^\infty}
\end{equation*}
so that
\begin{equation*}
\mathbb{P{'}}=\mathbb{P{}}-\mathbb{P{''}} \;\;\;\;\; \Rightarrow \;\;\;\;\; \left\lfloor{\raisebox{-.3pt}{\!{\dashuline{\begin{math}\,\mathbb{N}\end{math}}}}}\right\rfloor= \mathbb{P{}}-\left\lfloor{\raisebox{-.3pt}{\!{\dashuline{\begin{math}\,\mathbb{P}\end{math}}}}}\right\rfloor.
\end{equation*}
As we continue to sieve prime number subsequences of increasing order, we discover 
(\cite{OEIS_A333243}, \cite{OEIS_A333244}):

\smallskip

\noindent $\left\lfloor{\raisebox{-.3pt}{\!{\dashuline{\begin{math}\,{P}^{(2)} \end{math}}}}}\right\rfloor = $\noindent 3 \mycircle{5} 11 17 \mycircle{31} 41 \mycircle{59} 67 83 109 127 157 \mycircle{179} 191\dots
\smallskip

\noindent
and

\smallskip

\noindent $\left\lfloor{\raisebox{-.3pt}{\!{\dashuline{\begin{math}\,{P}^{(3)} \end{math}}}}}\right\rfloor = $\noindent 5 \mycircle{11} 31 59 \mycircle{127} 179 \mycircle{277} 331 431 599 709 919 \mybigcircle{1063} 1153\dots

\smallskip

which reveals the following relationships between sieved prime number sequences of higher order:
\begin{equation}
\left\lfloor{\raisebox{-.3pt}{\!{\dashuline{\begin{math}\,{P^{(2)}}\end{math}}}}} \right\rfloor = \mathbb{P}^{(3)}=\mathbb{P_{\mathbb{P_{{I}}}}} \label{eq:9}
\end{equation}
and
\begin{equation}
\left\lfloor{\raisebox{-.3pt}{\!{\dashuline{\begin{math}\,{P^{(3)}}\end{math}}}}} \right\rfloor = \mathbb{P}^{(4)}=\mathbb{P_{\mathbb{P_{\mathbb{P_{{I}}}}}}}. \label{eq:10}
\end{equation}
The sieving operations on the prime number subsequences of higher order in Eq. \ref{eq:9}--\ref{eq:10} can be further defined as
\begin{equation*}
\left\lfloor{\raisebox{-.3pt}{\!{\dashuline{\begin{math}\,{P^{(2)}}\end{math}}}}} \right\rfloor={P^{(2)}}-\mathbb{P}^{(2)}={P^{(2)}}-\left\lfloor{\raisebox{-.3pt}{\!{\dashuline{\begin{math}\,{P^{(1)}}\end{math}}}}} \right\rfloor
\end{equation*}
and
\begin{equation*}
\left\lfloor{\raisebox{-.3pt}{\!{\dashuline{\begin{math}\,{P^{(3)}}\end{math}}}}} \right\rfloor={P^{(3)}}-\mathbb{P}^{(3)}={P^{(3)}}-\left\lfloor{\raisebox{-.3pt}{\!{\dashuline{\begin{math}\,P^{(2)}\end{math}}}}}\right\rfloor.
\end{equation*}
With these new tools, we can now define all higher-order prime number subsequences of the form (see \cite{PrimesHavingPrimeSubscripts})
\begin{equation*}
\left\lbrace{p_{p_{._{._{._{p_{p_{n}}}}}}}}\right\rbrace{\Biggr\rvert_{n=1}^{\infty}} \\
\end{equation*}
in terms of the sieved prime number sequences as follows: 
\begin{align*}
\left\lbrace{p_{n}}\right\rbrace{\Biggr\rvert_{n=1}^{\infty}} &= P^{(1)}=\left\lfloor{\raisebox{-.3pt}{\!{\dashuline{\begin{math}\,\mathbb{N}\end{math}}}}}\right\rfloor + \left\lfloor{\raisebox{-.3pt}{\!{\dashuline{\begin{math}\,P^{(1)}\end{math}}}}} \right\rfloor = \mathbb{P}^{(1)}+\mathbb{P}^{(2)}
\\
\left\lbrace{p_{p_{n}}}\right\rbrace{\Biggr\rvert_{n=1}^{\infty}} &= P^{(2)} = \left\lfloor{\raisebox{-.3pt}{\!{\dashuline{\begin{math}\,P^{(1)}\end{math}}}}} \right\rfloor + \left\lfloor{\raisebox{-.3pt}{\!{\dashuline{\begin{math}\,P^{(2)}\end{math}}}}}\right\rfloor = \mathbb{P}^{(2)}+\mathbb{P}^{(3)}
\\
\left\lbrace{p_{p_{p_{n}}}}\right\rbrace{\Biggr\rvert_{n=1}^{\infty}} &= P^{(3)} = \left\lfloor{\raisebox{-.3pt}{\!{\dashuline{\begin{math}\,P^{(2)}\end{math}}}}}\right\rfloor + \left\lfloor{\raisebox{-.3pt}{\!{\dashuline{\begin{math}\,P^{(3)}\end{math}}}}}\right\rfloor = \mathbb{P}^{(3)}+\mathbb{P}^{(4)}
\\
\vdots
\end{align*}


\section{Venn Diagram Analysis}


\begin{figure}[h]
\centering
\includegraphics[scale=0.5]{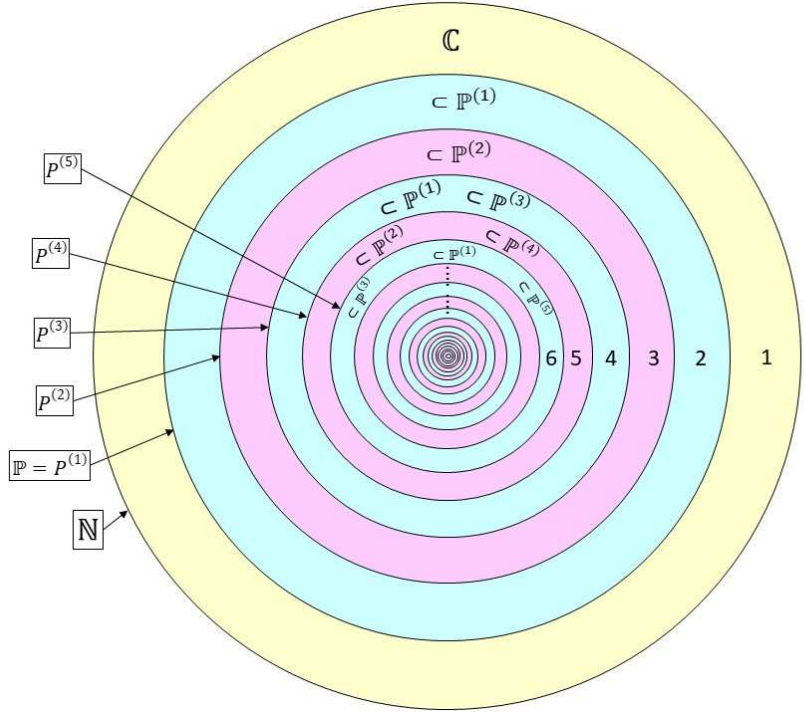}
\caption{Venn diagram relating the various sequences.} \label{F:Venn}
\end{figure}

Figure \ref{F:Venn} is a Venn diagram which illustrates the interrelationships between the prime number subsequences of higher order and the N-sieved prime number sequences heretofore disclosed.  

We interpret the diagram as follows: 
The outermost yellow ring \(\mathbb{C}\) is the composite numbers. 
Starting with the outermost blue ring, we see that the range of prime numbers $P^{(1)}-P^{(2)}$ is a subset of $\mathbb{P}^{(1)}$ as denoted by $\subset\mathbb{P}^{(1)}$ inscribed inside the outer blue ring.  Continuing on to the next blue ring inside the outer blue ring, we see that the range of prime numbers $P^{(3)}-P^{(4)}$ constitute the next subset of prime numbers to include in $\mathbb{P}^{(1)}$ denoted by $\subset\mathbb{P}^{(1)}$ inscribed inside that blue ring.  But the second inner blue ring in the Venn diagram not only contains the next subset of prime numbers to be included in the range of $\mathbb{P}^{(1)}$, but it also contains the first set of primes to be included in the next higher sieved set $\mathbb{P}^{(3)}$.  This pattern continues to infinity as each inner blue ring continues to contribute additional primes to include in $\mathbb{P}^{(1)}$ in addition to providing the beginning terms of a new sequence of a higher sieved set of primes which in turn will also continue to gather new prime numbers for its sequence as each blue ring is encountered all the way to infinity. 

This process works similarly for the alternate sets of sieved primes as follows: Starting with the outermost red ring, we see that the range of prime numbers $P^{(2)}-P^{(3)}$ is a subset of $\mathbb{P}^{(2)}$ as denoted by $\mathbb{P}^{(2)}$ inscribed inside the outer red ring.  Continuing on to the next red ring inside the outer red ring, we see that the range of prime numbers $P^{(4)}-P^{(5)}$ constitute the next subset of prime numbers to include in $\mathbb{P}^{(2)}$ denoted by $\subset\mathbb{P}^{(2)}$ inscribed inside that inner red ring.  But the second inner red ring not only contains the next subset of prime numbers to be included in the range of $\mathbb{P}^{(2)}$, but it also contains the first set of primes to be included in the next higher sieved set $\mathbb{P}^{(4)}$.  This pattern continues to infinity as each inner red ring continues to contribute additional primes to be included in $\mathbb{P}^{(2)}$ in addition to providing the beginning terms of a new sequence of a higher sieved set of primes which in turn will also continue to gather new prime numbers for its sequence as each inner red ring is encountered all the way to infinity.

In order to add more clarity to the Venn diagram, we define the elements of the first six annuli as follows:

\noindent Ring No. 1: $\mathbb{N}-\mathbb{P}=\mathbb{C}$

\noindent Ring No. 2: $P^{(1)}-P^{(2)} \subset \mathbb{P}^{(1)}$

\noindent Ring No. 3: $P^{(2)}-P^{(3)} \subset \mathbb{P}^{(2)}$

\noindent Ring No. 4: $P^{(3)}-P^{(4)} \subset \mathbb{P}^{(3)} \subset \mathbb{P}^{(1)} $

\noindent Ring No. 5: $P^{(4)}-P^{(5)} \subset \mathbb{P}^{(4)} \subset \mathbb{P}^{(2)} $

\noindent Ring No. 6: $P^{(5)}-P^{(6)} \subset \mathbb{P}^{(5)} \subset \mathbb{P}^{(3)} \subset \mathbb{P}^{(1)}$


\section{General Case}


At this point, we present a theorem for the general case of Theorem \ref{1.1}:

\begin{theorem} \label{6.1}
A unique subset of prime numbers $\mathbb{P}^{(i)}$ exists which, when applied as indexes to a subset of prime numbers $P^{(i)}$, yields the complement set of primes $\mathbb{P}^{(i+1)}$ so that
\begin{equation}
\mathbb{P}^{(i)}=P^{(i)}-\mathbb{P}^{(i+1)}. \label{eq:11}
\end{equation}
\end{theorem}

We will show that if there is a subset $\mathbb{P}^{(i)}\subset P^{(i)}$ such that $\mathbb{P}^{(i)}$ and $\mathbb{P}^{(i+1)}$ are a partition of $P^{(i)}$, that is if
\begin{equation*}
P^{(i)}=\mathbb{P}^{(i)}\cup\mathbb{P}^{(i+1)}
\end{equation*}
and
\begin{equation*}
\mathbb{P}^{(i)}\cap\mathbb{P}^{(i+1)}=0,
\end{equation*}
then $\mathbb{P}^{(i)}$ is unique and $\mathbb{P}^{(i)}$ and $\mathbb{P}^{(i+1)}$ are complement sets of  $P^{(i)}$.

\begin{proof}
We begin by showing that $\mathbb{P}^{(i)}$ of Eq. \ref{eq:11} is unique and that the first $n$ elements of $\mathbb{P}^{(i)}$ and $\mathbb{P}^{(i+1)}$ determine the $n+1st$ element of $\mathbb{P}^{(i)}$. First, we write
\begin{equation*}
q_1,\dots,q_n
\end{equation*}
to designate the smallest elements of $\mathbb{P}^{(i)}$ in increasing order. And because the $nth$ prime function $p_n$ is a monotonic increasing function, it follows that
\begin{equation*}
p_{q_{1}},\dots,p_{q_{n}}
\end{equation*}
are the smallest elements of $\mathbb{P}^{(i+1)}$ in corresponding increasing order. Now let $q$ be the smallest prime number that is different from all $q_1, \dots, q_n$ and $p_{q_{1}}, \dots, p_{q_{n}}$. We know that either $q \in \mathbb{P}^{(i)}$ or $q \in \mathbb{P}^{(i+1)}$, and our claim is that $q \in \mathbb{P}^{(i)}$. But if we suppose to the contrary that $q \in \mathbb{P}^{(i+1)}$, then there must be a prime number $r$ such that $q=p_r$. Now we assumed that $q>p_{q_{n}}$, and it follows from the definition that $r>q_n$ and $r<p_r=q$. But this means that $r$ satisfies the same conditions as $q$: i.e., that $r$ is the smallest prime different from all $q_1, \dots, q_n$ and $p_{q_{1}}, \dots, p_{q_{n}}$. So we should have chosen $r$ instead of $q$ because $r$ is smaller than the $q$ we assumed to be an element of $\mathbb{P}^{(i+1)}$. Therefore, our assumption that $q \in \mathbb{P}^{(i+1)}$ is a contradiction, and the smallest prime $q$ different from all $q_1$,\dots,$q_n$ and $p_{q_{1}}$,\dots,$p_{q_{n}}$ must be an element of $\mathbb{P}^{(i)}$ because the elements of $\mathbb{P}^{(i)}$, which serve as the subscripts of $\mathbb{P}^{(i+1)}$,  always lag behind the elements of $\mathbb{P}^{(i+1)}$. Furthermore, since all prime numbers $P^{(i)}$ are consumed in the process of adding the smallest available prime number different from all $q_1, \dots, q_n$ and $p_{q_{1}}, \dots, p_{q_{n}}$ to $\mathbb{P}^{(i)}$, then we have shown that $\mathbb{P}^{(i)}$ and $\mathbb{P}^{(i+1)}$ are complement sets of  $P^{(i)}$. Thus, we have shown by induction on $i$ that Theorem \ref{6.1}, as the general case of Theorem \ref{1.1}, holds true.
\end{proof}

We note here that the prime number sequences
\begin{equation*}
\left\lfloor{\raisebox{-.3pt}{\!{\dashuline{\begin{math}\,\mathbb{\mathbb{P^{'}}}\end{math}}}}}\right\rfloor,\left\lfloor{\raisebox{-.3pt}{\!{\dashuline{\begin{math}\,\mathbb{\mathbb{P^{''}}}\end{math}}}}}\right\rfloor,\left\lfloor{\raisebox{-.3pt}{\!{\dashuline{\begin{math}\,\mathbb{\mathbb{P^{'''}}}\end{math}}}}}\right\rfloor,\dots
\end{equation*}
do not exist because the complement set of primes corresponding to the indexes of the prime numbers in each of these sequences has already been sifted out via the N-sieving operation.  In other words, one cannot sift any terms out of a prime number sequence which has already been N-sieved.  As an example, consider the prime number subsequence
\begin{equation*}
\mathbb{P^{'}} = \left\lbrace {p{'}} \right\rbrace = \left\lbrace 2,5,7,13,19,23,29,31,37,43,47,53,59,61,71,\dots \right\rbrace.
\end{equation*}
If we attempt to apply the first term in this sequence as an index to the set of prime numbers for the purpose of eliminating that term from $\mathbb{P^{'}}$, we see that $p_2=3$ isn't present in $\mathbb{P^{'}}$.  Neither is the next term $p_5=11$.  And thus is the case for every term in $\mathbb{P^{'}}$ for which one would attempt to apply as an index to the set of all prime numbers for the purpose of eliminating primes from $\mathbb{P^{'}}$.


\section{Conclusion}


In this paper, we derived some of the basic properties of alternating sums of sequential sets of higher-order prime number subsequences and presented several methods available to derive them.  It is the author's hope that this research will shed some light on new ways to explore prime number sequences of higher order and that the methods presented herein will contribute new ideas to that end.

\section{Acknowledgements}

We gratefully acknowledge and thank the Number Theory Editor and the referees of the Missouri Journal of Mathematical Sciences for their helpful comments during the editing process, and we particularly express our gratitude to them for their help with the outlining of the proofs of the theorems.

\end{document}